\documentclass[12pt,a4paper]{amsart}

\usepackage{jipam}
\usepackage{amsmath,amssymb,amscd}
\usepackage{subfigure, epsfig}
\newcommand{\mb}{\mathbf}

\newcommand{\mc}{\mathcal}

\title[A Trace Inequality]{A trace inequality for positive definite matrices}
\author{Elena-Veronica Belmega}
\address{Universit\'e Paris-Sud XI\\ SUPELEC\\ Laboratoire des signaux et syst\`emes\\ Gif-sur-Yvette, France. }
\email{belmega@lss.supelec.fr}
\urladdr{http://veronica.belmega.lss.supelec.fr}

\author{Samson Lasaulce}
\address{CNRS\\ SUPELEC \\ Laboratoire des signaux et syst\`emes\\ Gif-sur-Yvette, France. }
\email{lasaulce@lss.supelec.fr}
\urladdr{http://samson.lasaulce.lss.supelec.fr}

\author{M\'erouane Debbah}
\address{SUPELEC\\ Alcatel-Lucent Chair on Flexible Radio\\ Gif-sur-Yvette, France. }
\email{merouane.debbah@supelec.com}
\urladdr{http://www.supelec.fr/d2ri/flexibleradio/debbah/}


\begin{document}

\begin{abstract}
In this note we prove that $\mathrm{Tr}\left\{ \mb{M} \mb{N} +
\mb{P} \mb{Q} \right\} \geq 0$ when the following two conditions are
met: (i) the matrices $\mb{M}, \mb{N}, \mb{P}, \mb{Q}$ are
structured as follows $\mb{M} = \mb{A} - \mb{B}$, $\mb{N} =
\mb{B}^{-1} - \mb{A}^{-1}$, $\mb{P} = \mb{C} - \mb{D}$, $\mb{Q} =
(\mb{B}+\mb{D})^{-1} -(\mb{A}+\mb{C})^{-1}$ (ii) $\mb{A}$, $\mb{B}$
are positive definite matrices and $\mb{C}$, $\mb{D}$ are positive
semidefinite matrices.
\end{abstract}

\maketitle

\section{Introduction}
\label{sec1}

Trace inequalities are useful in many applications. For example,
trace inequalities naturally arise in control theory (see e.g.,
\cite{lasserre-ac-1995}) and in communication systems with multiple
input and multiple output (see e.g., \cite{telatar-ett-1999}). In
this paper, the authors prove an inequality for which one
application has already been identified: the uniqueness of a pure
Nash equilibrium in concave games. Indeed, the reader will be able
to check that the proposed inequality allows one to generalize the
diagonally strict concavity condition introduced by Rosen in
\cite{rosen-eco-1965} to concave communication games with matrix
strategies \cite{belmega-wnc3-2008}.

Let us start with the scalar case. Let $\alpha, \beta, \gamma,
\delta$ be four reals such that $\alpha > 0, \beta >0 , \gamma \geq
0, \delta \geq 0$. Then, it can be checked that we have the
following inequality:
\begin{equation}
\label{equ1}
(\alpha-\beta)\left(\frac{1}{\beta}-\frac{1}{\alpha}\right) +
(\gamma-\delta)\left(\frac{1}{\beta+\delta}-\frac{1}{\alpha+\gamma}\right)
\geq 0.
\end{equation}
The main issue addressed here is to show that this inequality has a
matrix counterpart i.e., we want to prove the following theorem.
\begin{theorem}
\label{theorem} Let $\mb{A}$, $\mb{B}$ be two positive definite
matrices and $\mb{C}$, $\mb{D}$, two positive semidefinite matrices.
Then
\begin{equation}
\label{trace_2} \mc{T} = \mathrm{Tr}
\left\{(\mb{A}-\mb{B})(\mb{B}^{-1}-\mb{A}^{-1}) +
(\mb{C}-\mb{D})[(\mb{B} + \mb{D})^{-1}-(\mb{A}+\mb{C})^{-1}]\right\}
\geq 0.
\end{equation}
\end{theorem}
The closest theorem available in the literature corresponds to the
case $\mb{C} = \mb{D} = \mb{0}$, in which case the above theorem is
quite easy to prove. There are many proofs possible, the most simple
of them is probably the one provided by Abadir and Magnus in
\cite{abadir-book-2005}. In order to prove Theorem \ref{theorem} in
Sec. \ref{sec3} we will use some intermediate results which are
provided in the following section.

\section{Auxiliary Results}
\label{sec2}

Here we state three lemmas. The first two lemmas are available in
the literature and the last one is easy to prove. The first lemma is
the one mentioned in the previous section and corresponds to the case
$\mb{C} = \mb{D} = \mb{0}$.
\begin{lemma}\cite{abadir-book-2005}
\label{lemma_1} Let $\mb{A}$, $\mb{B}$ be two positive definite
matrices. Then
\begin{equation}
\label{trace_1} \mathrm{Tr}
\left\{(\mb{A}-\mb{B})(\mb{B}^{-1}-\mb{A}^{-1})\right\} \geq 0.
\end{equation}
\end{lemma}
The second lemma is very simple and can be found, for example, in
\cite{abadir-book-2005}. It is as follows.
\begin{lemma}\cite{abadir-book-2005}
\label{lemma_4} Let $\mb{M}$ and $\mb{N}$ be two positive
semidefinite matrices. Then
\begin{equation}
\mathrm{Tr}\{\mb{M}\mb{N}\} \geq 0.
\end{equation}
\end{lemma}
At last, we will need the following result.
\begin{lemma}
\label{lemma_3} Let $\mb{A}$, $\mb{B}$ be two positive definite
matrices, $\mb{C}$, $\mb{D}$, two positive semidefinite matrices
whereas $\mb{X}$ is only assumed to be Hermitian. Then
\begin{equation}
\mathrm{Tr} \left\{\mb{X}\mb{A}^{-1}\mb{X}\mb{B}^{-1}\right\} -
\mathrm{Tr}
\left\{\mb{X}(\mb{A}+\mb{C})^{-1}\mb{X}(\mb{B}+\mb{D})^{-1}\right\}
\geq 0.
\end{equation}
\end{lemma}
\begin{proof}
First note that $\mb{A}+\mb{C}\succeq \mb{A}$ implies (see e.g.,
\cite{horn-book-1991}) that $\mb{A}^{-1}\succeq(\mb{A}+\mb{C})^{-1}\succeq 0$ and that
$\mb{A}^{-1}-(\mb{A}+\mb{C})^{-1}\succeq
0$. In a
similar way we have $\mb{B}^{-1}-(\mb{B}+\mb{D})^{-1}\succeq 0$.
Therefore we obtain the following two inequalities:
\begin{equation}
\begin{array}{lcl}
\mathrm{Tr} \left\{\mb{X}\mb{A}^{-1}\mb{X}\mb{B}^{-1}\right\} &
\stackrel{(a)}{\geq}& \mathrm{Tr}
\left\{\mb{X}\mb{A}^{-1}\mb{X}(\mb{B}+\mb{D})^{-1}\right\}\\
\mathrm{Tr}
\left\{\mb{A}^{-1}\mb{X}(\mb{B}+\mb{D})^{-1}\mb{X}\right\}
&\stackrel{(b)}{\geq} & \mathrm{Tr}
\left\{(\mb{A}+\mb{C})^{-1}\mb{X}(\mb{B}+\mb{D})^{-1}\mb{X}\right\}
\end{array}
\end{equation}
where (a) follows by applying Lemma \ref{lemma_4} with
$\mb{M}=\mb{X}\mb{A}^{-1}\mb{X}$ and
$\mb{N}=\mb{B}^{-1}-(\mb{B}+\mb{D})^{-1} \succeq 0$ and (b) follows
by applying the same lemma with
$\mb{M}=\mb{A}^{-1}-(\mb{A}+\mb{C})^{-1} \succeq 0$ and
$\mb{N}=\mb{X}(\mb{B}+\mb{D})^{-1}\mb{X}$. Using the fact that
$\mathrm{Tr}
\left\{\mb{X}\mb{A}^{-1}\mb{X}(\mb{B}+\mb{D})^{-1}\right\}=
\mathrm{Tr}
\left\{\mb{A}^{-1}\mb{X}(\mb{B}+\mb{D})^{-1}\mb{X}\right\}$ we
obtain the desired result.
\end{proof}

\section{Proof of Theorem \ref{theorem}}
\label{sec3}

Let us define the auxiliary quantities $\mc{T}_1$ and $\mc{T}_2$ as
$\mc{T}_1 \triangleq \mathrm{Tr}
\left\{(\mb{A}-\mb{B})(\mb{B}^{-1}-\mb{A}^{-1})\right\}$ and
$\mc{T}_2 \triangleq \mathrm{Tr} \left\{(\mb{C}-\mb{D})[(\mb{B} +
\mb{D})^{-1}-(\mb{A}+\mb{C})^{-1}]\right\}$. Assuming $\mc{T}_2 \geq
0$ directly implies that $\mc{T} = \mc{T}_1 + \mc{T}_2 \geq 0$ since
$\mc{T}_1$ is always non-negative after Lemma \ref{lemma_1}. As a
consequence, we will only consider, from now on, the non-trivial
case where $\mc{T}_2 < 0$ (Assumption (A)).

First we rewrite $\mc{T}$ as:
\begin{equation}
\label{equ_1}
\begin{array}{lcl}
\mc{T}& =& \mathrm{Tr}
\left\{(\mb{A}-\mb{B})(\mb{B}^{-1}-\mb{A}^{-1})\right\} +
\mathrm{Tr} \left\{[(\mb{A}+\mb{C})-(\mb{B}+\mb{D})][(\mb{B} +
\mb{D})^{-1}-(\mb{A}+\mb{C})^{-1}]\right\} - \\
& & \mathrm{Tr} \left\{(\mb{A}-\mb{B})[(\mb{B} +
\mb{D})^{-1}-(\mb{A}+\mb{C})^{-1}]\right\} \\
& \stackrel{(c)}{\geq} &\mathrm{Tr}
\left\{(\mb{A}-\mb{B})(\mb{B}^{-1}-\mb{A}^{-1})\right\} -
 \mathrm{Tr} \left\{(\mb{A}-\mb{B})[(\mb{B} +
\mb{D})^{-1}-(\mb{A}+\mb{C})^{-1}]\right\} \\
& = & \mathrm{Tr}
\left\{(\mb{A}-\mb{B})\mb{B}^{-1}(\mb{A}-\mb{B})\mb{A}^{-1}\right\}
- \mathrm{Tr} \left\{(\mb{A}-\mb{B})(\mb{A}+\mb{C})^{-1}[(\mb{A} +
\mb{C})-(\mb{B}+\mb{D})](\mb{B} + \mb{D})^{-1}\right\} \\
& = & \mathrm{Tr}
\left\{(\mb{A}-\mb{B})\mb{B}^{-1}(\mb{A}-\mb{B})\mb{A}^{-1}\right\}
- \mathrm{Tr} \left\{(\mb{A}-\mb{B})(\mb{A}+\mb{C})^{-1}(\mb{A} -
\mb{B})(\mb{B} + \mb{D})^{-1}\right\} - \\
& &  \mathrm{Tr} \left\{(\mb{A}-\mb{B})(\mb{A}+\mb{C})^{-1}(\mb{C}
-\mb{D})(\mb{B} + \mb{D})^{-1}\right\}
\end{array}
\end{equation}
where (c) follows from Lemma \ref{lemma_1}. We see from the last
equality that if we can prove that $\mathrm{Tr}
\left\{(\mb{A}-\mb{B})(\mb{A}+\mb{C})^{-1}(\mb{C} -\mb{D})(\mb{B} +
\mb{D})^{-1}\right\} \leq 0$, proving $\mc{T} \geq 0$ boils down to
showing that
\begin{equation}
\mc{T}' \triangleq \mathrm{Tr}
\left\{(\mb{A}-\mb{B})\mb{B}^{-1}(\mb{A}-\mb{B})\mb{A}^{-1}\right\}
- \mathrm{Tr} \left\{(\mb{A}-\mb{B})(\mb{A}+\mb{C})^{-1}(\mb{A} -
\mb{B})(\mb{B} + \mb{D})^{-1}\right\} \geq 0.
\end{equation}
Let us show that $\mathrm{Tr}
\left\{(\mb{A}-\mb{B})(\mb{A}+\mb{C})^{-1}(\mb{C} -\mb{D})(\mb{B} +
\mb{D})^{-1}\right\} \leq 0$. By assumption we have that
$\mathrm{Tr} \left\{(\mb{C}-\mb{D})[(\mb{B} +
\mb{D})^{-1}-(\mb{A}+\mb{C})^{-1}]\right\} < 0$ which is equivalent
to
\begin{equation}
\mathrm{Tr} \left\{(\mb{A}-\mb{B})[(\mb{B} +
\mb{D})^{-1}-(\mb{A}+\mb{C})^{-1}]\right\} > \mathrm{Tr}
\left\{[(\mb{A}+\mb{C})-(\mb{B}+\mb{D})][(\mb{B} +
\mb{D})^{-1}-(\mb{A}+\mb{C})^{-1}]\right\}.
\end{equation}
From this inequality and Lemma \ref{lemma_1} we have that
\begin{equation}
\mathrm{Tr} \left\{(\mb{A}-\mb{B})[(\mb{B} +
\mb{D})^{-1}-(\mb{A}+\mb{C})^{-1}]\right\} > 0.
\end{equation}
On the other hand, let us rewrite $\mc{T}_2$ as
\begin{equation}
\label{equ_3}
\begin{array}{lcl}
\mc{T}_2 & = & \mathrm{Tr} \left\{(\mb{C}-\mb{D})(\mb{B} +
\mb{D})^{-1}[(\mb{A} -
\mb{B})+(\mb{C}-\mb{D})](\mb{A}+\mb{C})^{-1}\right\}\\
 & = & \mathrm{Tr} \left\{(\mb{C}-\mb{D})(\mb{B} + \mb{D})^{-1}(\mb{A} -
\mb{B})(\mb{A}+\mb{C})^{-1}\right\} + \\
& & \mathrm{Tr} \left\{(\mb{C}-\mb{D})(\mb{B} + \mb{D})^{-1}
(\mb{C}-\mb{D})(\mb{A}+\mb{C})^{-1}\right\} \\
& = &  \mathrm{Tr} \left\{(\mb{C}-\mb{D})(\mb{B} +
\mb{D})^{-1}(\mb{A} - \mb{B})(\mb{A}+\mb{C})^{-1}\right\} +
\mathrm{Tr} [\mb{Y}\mb{Y}^H]
\end{array}
\end{equation}
where $\mb{Y}= (\mb{A}+\mb{C})^{-1/2}(\mb{C}-\mb{D})(\mb{B} +
\mb{D})^{-1/2}$. Thus $\mc{T}_2 <0$ implies that:
\begin{equation}
\mathrm{Tr} \left\{(\mb{C}-\mb{D})(\mb{B} + \mb{D})^{-1}(\mb{A} -
\mb{B})(\mb{A}+\mb{C})^{-1}\right\} <0,
\end{equation}
which is exactly the desired result since $ \mathrm{Tr}
\left\{(\mb{C}-\mb{D})(\mb{B} + \mb{D})^{-1}(\mb{A} -
\mb{B})(\mb{A}+\mb{C})^{-1}\right\} = \mathrm{Tr}
\left\{(\mb{A}-\mb{B})(\mb{A}+\mb{C})^{-1}(\mb{C} -\mb{D})(\mb{B} +
\mb{D})^{-1}\right\} $. In order ton conclude the proof we only need
to prove that $\mc{T}' \geq 0$. This is ready by noticing that
$\mc{T}'$ can be rewritten as $\mc{T}' \triangleq \mathrm{Tr}
\left\{(\mb{A}-\mb{B})\mb{A}^{-1} (\mb{A}-\mb{B})\mb{B}^{-1}\right\}
- \mathrm{Tr} \left\{(\mb{A}-\mb{B})(\mb{A}+\mb{C})^{-1}(\mb{A} -
\mb{B})(\mb{B} + \mb{D})^{-1}\right\}$ and calling for Lemma
\ref{lemma_3} with $\mb{X} = \mb{A} - \mb{B}$, concluding the proof.

\end{document}